\documentclass[11pt]{article} 

\usepackage[utf8]{inputenc} 
\usepackage{geometry} 
\usepackage{graphicx} 
\usepackage{authblk}

\usepackage{amsmath,amsthm,amsfonts}
\usepackage{marvosym}
\usepackage{graphicx}
\usepackage{color}
\usepackage{epsfig}
\usepackage{fullpage}
\theoremstyle{plain}
\newtheorem{thm}{Theorem}[section]
\newtheorem{cor}[thm]{Corollary}
\newtheorem{defi}[thm]{Definition}
\newtheorem{conj}[thm]{Conjecture}
\newtheorem{prop}[thm]{Proposition}
\newtheorem{lemma}[thm]{Lemma}

\newtheorem*{example}{Example}
\newtheorem*{claim*}{Claim A1}
\newtheorem*{claim2*}{Claim 1}
\newtheorem*{claim3*}{Claim 2}

\newtheorem*{thior1}{Theorem A}
\newtheorem*{thior2}{Theorem B}
\newtheorem*{thior3}{Theorem C}

\theoremstyle{definition}

\newtheorem{remark}[thm]{Remark}

\newcommand{\comment}[1]{}

\usepackage{booktabs} 
\usepackage{array} 
\usepackage{paralist} 
\usepackage{verbatim} 
\usepackage{subfig} 

\usepackage{fancyhdr} 
\usepackage{sectsty}
\allsectionsfont{\sffamily\mdseries\upshape} 

\usepackage[nottoc,notlof,notlot]{tocbibind} 
\usepackage[titles,subfigure]{tocloft} 

\title{On lattice cohomology and left-orderability}
\author{Mauro Mauricio \\ Alfr\'ed R\'enyi Institute of Mathematics \\ \vskip0.1in email: \textit{mauro.mauricio@imperial.ac.uk}}
\date{\today}

\begin{document}
\maketitle

\begin{abstract} It has been recently conjectured by Boyer-Gordon-Watson \cite{BGW} that a closed, orientable, irreducible $3$-manifold $M$ is a Heegaard Floer $L$-space if and only if $\pi_1(M)$ is not left-orderable. In this article, we study this conjecture from the point of view of lattice cohomology, an invariant introduced by N\'emethi in \cite{Nem1} which is conjecturally isomorphic to the $HF^+$ version of Heegaard Floer homology. Using the invariant's combinatorial tractability as a stepping stone, we produce some interesting quite general families of negative-definite graph manifolds against which to test the Boyer-Gordon-Watson conjecture. Then, using horizontal foliation arguments and direct manipulation of the fundamental group, we prove that these families do indeed satisfy the conjecture.

\end{abstract}

\section{Introduction}

One of the interesting present challenges of $3$-manifold theory is to relate  Ozsv\'ath and Szab\'o's celebrated Heegaard Floer invariants \cite{OS2} to the fundamental group. Until recently, the only known connection (a very tenuous one) was an observation of Yi Ni \cite{Ni}, who remarked that knot Floer homology $\widehat{HFK}$ detects whether or not the fundamental group of a knot complement has finitely generated commutator subgroup, something one obtains by putting together the deep fact that $\widehat{HFK}$ detects fibred knots (the main result of \cite{Ni}) with the Stallings fibration theorem. Since then, a more substantial connection between these two topics seems to have appeared. Recall that a Heegaard Floer homology $L$-space is a closed oriented rational homology sphere $M$ which satisfies $HF^+_{red}(M,\mathfrak{t})=0$ for all $\mathfrak{t}\in Spin^c(M)$. Among the class of rational homology spheres, $L$-spaces have the simplest Heegaard Floer invariants. Also, recall that a nontrivial group is said to be left-orderable if and only if it admits a strict total ordering that is invariant under left multiplication (as a matter of convenience, we declare the trivial group to be non-left-orderable).  In \cite{BGW}, Boyer-Gordon-Watson propose the following connection between Heegaard Floer homology and the fundamental group:

\begin{conj} \label{LO} A closed, orientable, irreducible rational homology sphere $M$ is a Heegaard Floer $L$-space if and only if $\pi_1(M)$ is not left-orderable.
\end{conj}

\noindent There is a growing body of evidence pointing towards the validity of Conjecture \ref{LO}: it is known to hold for $Sol$ manifolds and Seifert-fibered spaces  \cite{BGW}, and also for graph-manifold integral homology spheres \cite{BB}. Moreover, there are examples and some results in the realm of hyperbolic manifolds which support the conjecture, mostly obtained by surgery on knots in $S^3$ \cite{Li, Tran}.

\vskip 0.1in It seems plausible that taut foliations constitute the geometry behind Conjecture \ref{LO}: the conditions in the conjecture may be equivalent to the non-existence of co-orientable taut foliations. Recall that a codimension-one foliation $\mathcal{F}$ of a manifold $M$ is said to be taut if every leaf of $\mathcal{F}$ meets a loop which is transverse to $\mathcal{F}$. It is well-known that if $M$ admits a co-orientable taut foliation, then $M$ cannot be a Heegaard Floer $L$-space \cite{OS1}. Moreover, it is known that if $M$ admits a co-orientable taut foliation, then the commutator subgroup $[\pi_1(M), \pi_1(M)]$ is left-orderable \cite{BRW}. In addition to this, in all the special cases described above for which Conjecture \ref{LO} is solved, the connection with taut foliations is also established.

\vskip 0.1in In this article, we seek to add to this evidence. We will consider manifolds which are given by plumbing along a negative-definite weighted tree $\Gamma$. We will denote such manifolds also by $\Gamma$, in an abuse of notation (this should cause no confusion). We will also require that the tree $\Gamma$ be minimal, that is, no valence-two or valence-one vertices should carry weight equal to $-1$ (there is no loss in generality in demanding this, since we can always blow down to bring $\Gamma$ into minimal form). Our strategy is to use lattice cohomology $\mathbb{H}(\Gamma)$, an invariant of negative-definite plumbings introduced by N\'emethi in \cite{Nem1}, which is conjecturally isomorphic to the $HF^+$ version of Heegaard Floer homology. In this theory, deciding whether $\Gamma$ is a ``lattice cohomology $L$-space'' (\textit{i.e.} whether $\mathbb{H}_{red}=0$) or not can be done fairly simply via a combinatorial procedure known as Laufer's algorithm \cite{Laufer} (for the reader unfamiliar with this, consult the next section for more details). Using this combinatorial nature of lattice cohomology, it is fairly easy to construct interesting families of manifolds $\Gamma$ for which the decision process above is particularly quick and straightforward. For some large such families, we study the left-orderability of their fundamental group, and prove that they support Conjecture \ref{LO}. In the following, we introduce the lexicon required in part to state our main results:

\begin{defi} Let $\Gamma$ be a weighted graph, and let $v\in \Gamma$ be a vertex. Let $w(v)$ be the integer weight associated with $v$, and $n(v)$ to be the number of neighbours of $v$. Then

\begin{enumerate}

\item The quantity $d(v)=n(v)-|w(v)|$ is said to be the deficiency of $v$.
\item A vertex $v$ such that $n(v)=1$ is said to be a leaf.
\item A vertex $v$ such that $n(v)=2$ is said to be a bamboo.
\item A vertex $v$ such that $n(v)>2$ is said to be a node.
\item If $d(v)\leq0$, $v$ is said to be good.
\item If $d(v)\geq 1$, $v$ is said to be bad.
\item It $d(v)\geq 2$, $v$ is said to be very bad.

\end{enumerate}
\end{defi} 

\begin{defi} \label{insu} Suppose that $\Gamma$ is a negative-definite plumbing tree with no very bad vertices, subject to the following conditions:

\begin{enumerate}
\item  Every good vertex $v$ satisfies $d(v)+K(v)\leq 0$, where $K(v)$ is the number of bad vertex neighbours of $v$. 
\item No two bad vertices are neighbours.
\end{enumerate}

\noindent Then, $\Gamma$ is said to be insulated.

\end{defi}

\noindent Let us emphasise once again that we only concern ourselves with negative-definite trees. Our main results are the following:

\begin{thior1}\label{A} If $\Gamma$ has a very bad vertex, then $\pi_1(\Gamma)$ is left-orderable. Moreover, $\Gamma$ admits a co-orientable taut foliation, and therefore is not a Heegaard Floer homology L-space.
\end{thior1}

\begin{thior2} \label{B} Suppose that $\Gamma$ contains a proper negative-definite $E_8$ subgraph. Then, $\pi_1(\Gamma)$ is left-orderable. Moreover, $\Gamma$ admits a co-orientable taut foliation, and therefore is not a Heegaard Floer homology L-space.
\end{thior2}

\begin{remark} These two results are false if one drops the negative-definite condition, as can be seen in the examples below.

\begin{center}
\includegraphics[scale=0.3]{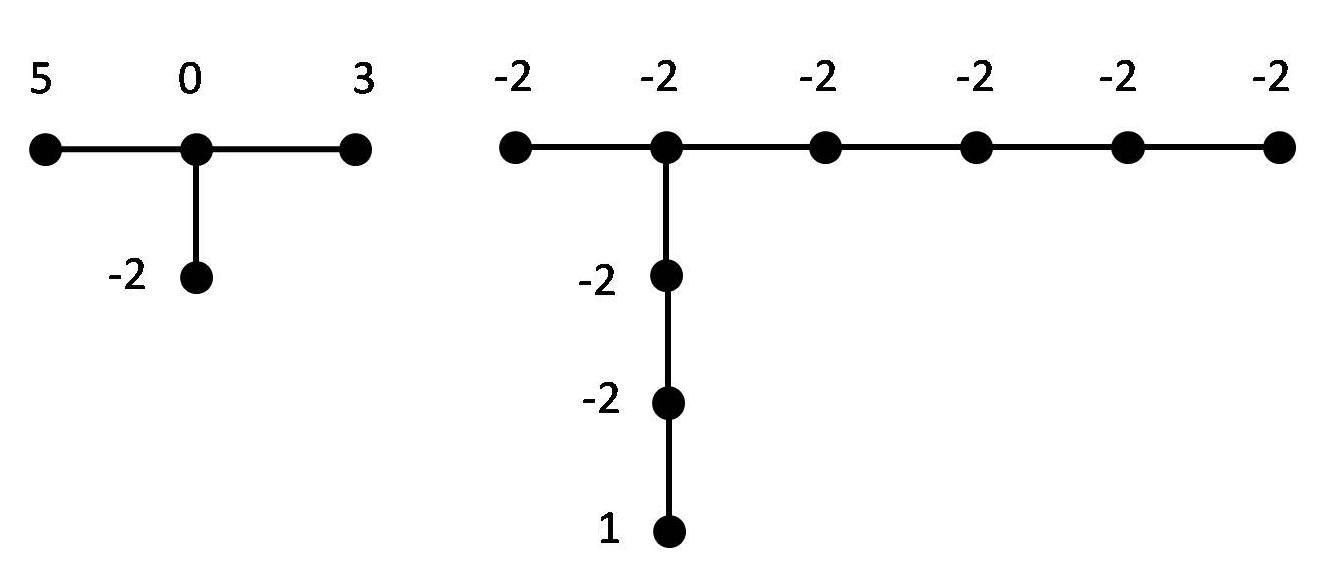}
\end{center}

\noindent The example on the left represents the Poincar\'e homology sphere $\Sigma(2,3,5)$, and it is well-known that this manifold is an $L$-space and that its fundamental group is not left-orderable (this group is finite, and it is easy to see that left-orderable groups must be torsion-free). For the example on the right, we can blow down the $+1$ leaf to obtain a negative-definite configuration, and it can be checked easily that this is a Heegaard Floer homology $L$-space, and that its fundamental group is not left-orderable (see Theorem \ref{stipi} in the next section).
\end{remark}

On the converse statement in Conjecture \ref{LO}, we prove the following:

\begin{thior3} \label{C} Suppose that $\Gamma$ is insulated. Then, $\pi_1(\Gamma)$ is not left-orderable. Moreover, $\Gamma$ is a Heegaard Floer $L$-space. 
\end{thior3}

\begin{example} The following graph is insulated:
\end{example}

\begin{figure}[h!]
\centering
\includegraphics[scale=0.30]{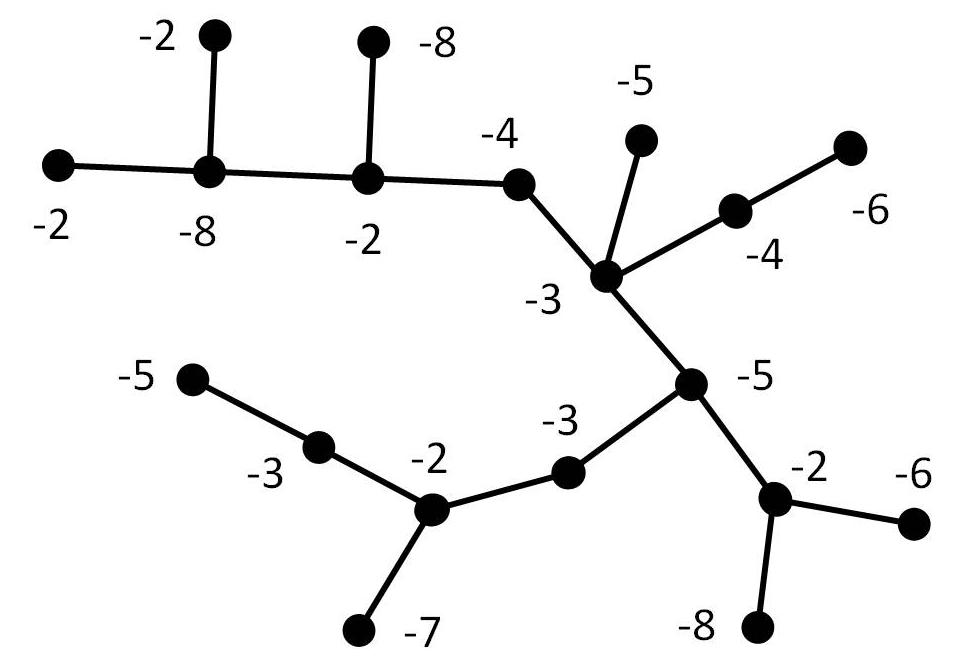}
\end{figure}

\begin{remark} It is perhaps worth pondering upon the strength of Theorem C \textit{i.e.} how restrictive the insulation condition is. It is easy to see that if two bad vertices are neighbours, then $\Gamma$ is not a lattice cohomology $L$-space, so that if one drops condition \textit{2} in Definition \ref{insu}, and if one believes in Conjecture \ref{LO} and that lattice cohomology is isomorphic to Heegaard Floer homology, the conclusion of Theorem C should no longer hold. Moreover, if we have a good vertex $x$ satisfying $d(v)+K(v)\geq 2$, then $\Gamma$ is also not a lattice cohomology $L$-space, so similar remarks apply. In the ``critical'' case $d(v)+K(v)=1$, the answer can be extremely subtle and ``non-local'', in the sense that different regions of the graph interact in a possibly very complicated manner, so that it becomes impossible to characterise lattice cohomology $L$-spaces in terms of local conditions on the vertices. Thus, we like to think of the insulation condition as somehow the most general local condition that can be put on the vertices to ensure that $\Gamma$ is a lattice cohomology $L$-space.  
\end{remark}

\vskip 0.1in This article is organised as follows: in section 2, we introduce the relevant background in lattice cohomology, with an emphasis on how to detect lattice cohomology $L$-spaces. We also recall several previously known results on left-orderability of $3$-manifold groups and taut foliations, which will be crucial for the proofs of theorems A and B. Finally, in section 3, we prove our results stated in the introduction.

\section*{Acknowledgements}

 We would like to thank Andr\'as N\'emethi and Andr\'as Stipsicz for inspiring conversations and for their interest in this work. Moreover, we would like to acknowledge the work of Clay-Lidman-Watson in  \cite{CLW}: even though we do not directly use any of their results (since then, stronger results have appeared in \cite{BB}), their paper served as an important source of inspiration for our present work, and was the crux of an earlier version. Finally, financial support from the  ERC LTDBud grant and
from the Lend\"ulet grant of the Hungarian Academy of Sciences is also gratefully acknowledged.

\section{Background Material}

\subsection{Plumbings and lattice cohomology} 

Let $\Gamma$ be a tree (not necessarily negative-definite) with $N$ vertices, $\{i\}_{i=1}^N$, such that each vertex $i$ is decorated with a pair of integers $(g_i,e_i)$, $g_i\geq 0$. This datum determines a closed, oriented $3$-manifold via the following construction: for each vertex $i$, take a  circle bundle over a closed, orientable surface of genus $g_i$ and with Euler number $e_i$. From each bundle, remove the regular neighbourhood of $n(i)$ circle fibers (recall that $n(i)$ is the number of neighbours of $i$). The result is a collection of bundles over surfaces of genera $g_i$, punctured $n(i)$ times, $F_{g_i}^{n(i)}$. It is a standard fact that these are homeomorphic to $F_{g_i}^{n(i)}\times S^1$. In particular, we can equip each (torus) boundary component $j$ of the $i$th punctured bundle with a canonical basis for its first homology, $\{\mu_{i,j}, \lambda_i\}$, where $\mu_{i,j}$ is the relevant component of $\partial F_{g_i}^{n(i)} $, and $\lambda_i$ is a circle fiber. We then glue two punctured bundles along boundary components if and only if the corresponding vertices are joined by an edge in $\Gamma$. On each  boundary component, this gluing is determined by a matrix $\phi\in GL(2,\mathbb{Z})$, once a basis for the first homology is chosen. With the choice of basis described above, the gluing matrix is taken to be

$$ \phi=\left( \begin{array}{cc}
0 & 1 \\
1 & 0 \\
 \end{array} \right) $$

\noindent for every boundary component. The manifold that results from this operation is called the plumbing along $\Gamma$, and is denoted also by $\Gamma$, as stated in the introduction. In particular, there is an associated intersection matrix, also denoted $\Gamma=(\gamma_{ij})$, whose diagonal entries are the Euler numbers $e_i$, and the off-diagonal entries satisfy $\gamma_{ij}=1$ if $i$ and $j$ are neighbours, $\gamma_{ij}=0$ otherwise. Now, it is an easy standard exercise in algebraic topology to show that the manifold $\Gamma$ is a rational homology sphere if and only if all the genera $g_i$ are equal to zero and the matrix $\Gamma$ is nonsingular. If that is the case, we omit the $g_i$ from the decoration of the tree $\Gamma$, and this becomes a graph with each vertex decorated by a single integer weight, $e_i$. Moreover, notice that the intersection matrix determines the manifold $\Gamma$.

\begin{remark} In terms of JSJ decompositions, plumbings correspond to manifolds whose JSJ pieces consist entirely of Seifert-fibered manifolds. In particular, it is easy to see the JSJ tori directly from a minimal graph $\Gamma$: they are in one-to-one correspondence with edges separating nodes of $\Gamma$.
\end{remark}

We now describe the elements of lattice cohomology that will be relevant to our work. This is not intended to be an introduction to the subject (we do not even present the definition), we merely expose the absolute minimum background required for understanding Laufer's algorithm. As stated in the introduction, lattice cohomology is an invariant introduced by N\'emethi in \cite{Nem1}, inspired by the theory of complex normal surface singularities: to a negative-definite plumbing graph $\Gamma$, one associates a topological invariant in the guise of a $\mathbb{Z}[U]$-module $\mathbb{H}(\Gamma)$, computed from the plumbing data. Just like Heegaard Floer homology, it splits as a direct sum over the $Spin^c$ structures of the manifold $\Gamma$. Moreover, in the case where $\Gamma$ is a rational homology sphere, we have one further splitting for every $Spin^c$ structure $\mathfrak{t}$ in $\Gamma$,

$$\mathbb{H}(\Gamma,\mathfrak{t})=\mathcal{T}^+\oplus \mathbb{H}_{red}(\Gamma,\mathfrak{t})$$

\noindent where $\mathcal{T}^+\simeq \mathbb{Z}[U,U^{-1}]/U\cdot\mathbb{Z}[U]$ and $\mathbb{H}_{red}$ is $\mathbb{Z}[U]$-torsion. Thus, in analogy with Heegaard Floer theory, we make the following definition:

\begin{defi} Suppose that $\Gamma$ is a rational homology sphere. If \;$\mathbb{H}_{red}(\Gamma,\mathfrak{t})=0$ for all $\mathfrak{t}\in Spin^c(\Gamma)$, we say that $\Gamma$ is a lattice cohomology $L$-space.
\end{defi}

\noindent When $\Gamma$ is a rational homology sphere, it is conjectured \cite{Nem1} that $\mathbb{H}(\Gamma,\mathfrak{t})\simeq HF^+(\Gamma,\mathfrak{t})$, for any $\mathfrak{t}\in Spin^c(\Gamma)$. This isomorphism has actually been verified in several instances, the following special case being of interest to us:

\begin{thm}\label{isoiso} (\cite{Nem2}) Suppose that $\Gamma$ is a lattice cohomology $L$-space. Then, $\Gamma$ is a Heegaard Floer $L$-space.
\end{thm}

 The great advantage of lattice cohomology over Heegaard Floer homology is its computability. Unlike the latter theory, whose calculation involves essentially solving a PDE, the former can be computed combinatorially. In particular, there is a very simple criterion for characterising lattice cohomology $L$-spaces, given by Laufer's algorithm, which we now describe: associated to any plumbing we have a lattice, that is, a symmetric bilinear form

$$ \langle \;,\; \rangle: \mathbb{Z}^N \otimes \mathbb{Z}^N\longrightarrow \mathbb{Z}
$$

\noindent which is given, in some basis represented by the individual vertices, $\{E_i\}_{i=1}^N$, by the matrix $\Gamma$. For each vertex $v\in \Gamma$, let $E_v$ be the associated basis element. Define the canonical vector $K:=\sum_{v\in \Gamma} {k_v}E_v$, where $k_v=-\langle E_v , E_v \rangle-2$. Then, there is a weight function $\chi: \mathbb{Z}^N \longrightarrow \mathbb{Z}$ given by

$$\chi(l)=-\frac{\langle l+K,l \rangle}{2}
$$

\noindent The algorithm is then given by the following:

\begin{enumerate}
\item $z_0:=\sum_{v\in \Gamma} E_v $
\item If, at step $i$, $\exists$ $v(i)\in \Gamma$ such that $\langle z_i, E_{v(i)} \rangle>0$, define $z_{i+1}:=z_i+ E_{v(i)}$. If there is more than one $v(i)$ satisfying this condition, choose one arbitrarily.
\item If, at step $t$, $\langle z_t ,E_v \rangle\leq 0$ for all vertices $v\in \Gamma$, define $z_{min}:=z_t$. The procedure terminates here.
\end{enumerate}

\noindent It was shown by Laufer \cite{Laufer} that the above procedure always terminates (provided $\Gamma$ is negative-definite), and that the output $z_{min}$ is independent of the choice of path in the algorithm. In the singularity theory literature $z_{min}$ is known as Artin's minimal cycle, and it plays an important role in the algebro-geometric study of normal surface singularities. For the purposes of this article, though, we will content ourselves with the topological significance of $z_{min}$, which is due to N\'emethi \cite{Nem2}:

\begin{thm}\label{nemi} $\Gamma$ is a lattice cohomology $L$-space if and only if $\chi(z_{min})=1$.
\end{thm}

\noindent It is easy to show that $\chi(z_0)=1$ and that $\chi(z_i)\geq\chi(z_{i+1})$, so that if we notice at any step of the algorithm that $\chi$ has dropped, then we may stop there and conclude that $\Gamma$ is not a lattice cohomology $L$-space.

\vskip 0.1in For the reader's convenience, a more conceptual and equivalent way to think about Laufer's algorithm together with Theorem \ref{nemi}, which is essentially just unpacking the function $\chi$ in terms of the adjacency matrix, is informally described  in the following: 

\begin{enumerate}
\item For each vertex $v\in\Gamma$, compute the deficiency $d(v)$.
\item For each vertex $v$ such that $d(v)\geq1$, update the deficiency by adding $w(v)$ to $d(v)$. Moreover, increase the deficiency of each of the neighbours of $v$ by $1$.
\item Repeat the above two steps (each time computing the updated deficiency) until either a) a vertex of deficiency $\geq 2$ appears or b) all vertices have nonpositive deficiency.
\item In case a) above, $\Gamma$ is not a lattice cohomology $L$-space. In case b) above, $\Gamma$ is a lattice cohomology $L$-space.
\end{enumerate}

Despite its apparent simplicity, Laufer's algorithm can showcase some very subtle behaviour. The reader is challenged to test it on the next example, lest the problem seem too straightforward.

\begin{example} Recall the negative-definite $E_8$ graph, which we present below, together with a labelling of the vertices. This graph represents the Poincar\'e homology sphere.

\begin{figure}[h!]
\centering \label{steps}
\includegraphics[scale=0.30]{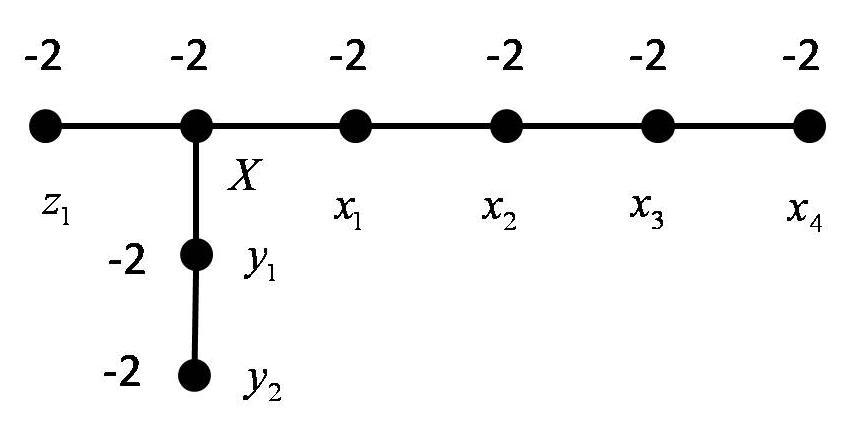}
\caption{The negative-definite $E_8$ graph.}
\end{figure}

\noindent After a staggering 21 iterations, we obtain $z_{min}=6X+5x_1+4x_2+3x_3+2x_4+3z_1+4y_1+2y_2$, and $\chi(z_{min})=1$, confirming that the Poincar\'e homology sphere is a lattice cohomology $L$-space, as expected.

\end{example}

As a start to our investigation of the left-orderability conjecture, we prove the next result:

\begin{prop}\label{lattice} We have the following:

\begin{enumerate}
\item If $\Gamma$ has a very bad vertex, it is not a lattice cohomology $L$-space.

\item If $\Gamma$ contains a proper $E_8$ subgraph, it is not a lattice cohomology $L$-space.

\item If $\Gamma$ is insulated, it is a lattice cohomology $L$-space.
\end{enumerate}
\end{prop}

\begin{proof} Statement \textit{1} is immediate. To see statement \textit{2}, first note that if there are two bad vertices either as neighbours or connected by a bamboo chain with weights $-2$, then $\Gamma$ is not a lattice cohomology $L$-space. Moreover, it is a standard fact that if we have $\Gamma'\subset \Gamma$, with minimal cycles $z'_{min}$ and $z_{min}$ respectively, then $\chi(z_{min})\leq \chi(z'_{min})$. Therefore, it suffices to consider the case where there is a single vertex abutting from one of the leaves of the $E_8$ graph. We will use the labelling of the vertices displayed in the figure above. If there is a vertex $w$ abutting from $y_2$, after 10 steps we see $z_{10}=3X+3x_1+2x_2+2x_3+x_4+2z_1+3y_1+2y_2+w$ (this is not the minimal cycle), and it can be computed that $\chi(z_{10})<1$. So we may stop here. If there is a vertex $w$ abutting from $z_1$, then we may stop at $z_{4}=2X+2x_1+x_2+x_3+x_4+2z_1+w+2y_1+y_2$, where we see $\chi(z_{4})<1$. Finally, if there is $w$ abutting from $x_4$, we terminate at $z_{20}=5X+5x_1+4_2+3x_3+2x_4+w+3z_1+4y_1+3y_2$, also with $\chi(z_{20})<1$, thereby finishing the proof of statement \textit{2}. Statement \textit{3} is immediate from the ``informal'' description of the algorithm presented above: we do one iteration for each bad vertex, and since each good vertex satisfies  $d(v)+K(v)\leq 0$, we do not need to perform any further iterations.
\end{proof}

\begin{remark} The reader is invited to compare the computation in the third subcase in the proof of statement \textit{2} with the one for the $E_8$ graph to appreciate the subtlety in Laufer's algorithm.
\end{remark}

Theorem \ref{isoiso} immediately implies the next result:

\begin{cor}\label{lati} If $\Gamma$ is insulated, it is a Heegaard Floer $L$-space.
\end{cor}

\begin{remark} Since we will be showing theorems A and B through taut foliations, statements \textit{1} and \textit{2} are somehow not needed in hindsight. However, we decided to include them to emphasise that lattice cohomology was the dynamo of our thought process, and that we certainly wouldn't have discovered our results without it.
\end{remark}

\subsection{Left-orderable $3$-manifold groups and taut foliations}

In this subsection, we discuss some known facts about left-orderable groups and taut foliations. Recall the definition of a left-orderable group:

\begin{defi}Let $G$ be a nontrivial group. We say that $G$ is left-orderable if there exists a strict total ordering $>$ on $G$ with the property that, given any $g,h\in G$ such that $g>h$, then $fg>fh$, for all $f\in G$.
\end{defi}

\noindent The study of orderable $3$-manifold groups was undertaken systematically in a paper of Boyer-Rolfsen-Wiest \cite{BRW}. In it, the following characterisation is proved:

\begin{thm} Let $M$ be a compact, irreducible, $P^2$-irreducible $3$-manifold. Then, $\pi_1(M)$ is left-orderable if only if there exists a nontrivial homomorphism $\phi: \pi_1(M)\longrightarrow L$, where $L$ is any left-orderable group.
\end{thm}

\noindent We remind the reader that we declare the trivial group to be non-left-orderable. An immediate consequence of the above result is that a compact, irreducible, $P^2$-irreducible $3$-manifold $M$ with $b_1(M)>0$ has left-orderable fundamental group (take the Hurewicz homomorphism $Hz: \pi_1(M)\longrightarrow H_1(M;\mathbb{Z})$, then project onto $\mathbb{Z}$, and notice that the latter group is left-orderable via the standard ordering as a subset on $\mathbb{R}$).

\vskip 0.1in
As mentioned in the introduction, taut foliations are important in the study of Conjecture \ref{LO}. Recall that a codimension-one foliation $\mathcal{F}$ of $M$ is said to be $\mathbb{R}$-covered if the leaf space of the induced foliation $\widetilde{\mathcal{F}}$ on the universal cover $\widetilde{M}$ is homeomorphic to the real line. If an $\mathbb{R}$-covered foliation is also co-orientable, then the action of $\pi_1(M)$ as deck transformations gives us a representation into the group of orientation-preserving homeomorphisms of the real line, $\rho: \pi_1(M)\longrightarrow Homeo^+(\mathbb{R})$. It is well-known that $Homeo^+(\mathbb{R})$ is left-orderable. Thus, if one can show that $\rho$ is nontrivial, it follows that $\pi_1(M)$ is left-orderable. It turns out that very often it is possible to upgrade a taut foliation to a co-orientable $\mathbb{R}$-covered foliation. One particular success story happens in the realm of Seifert manifolds: a codimension-one foliation of a Seifert manifold is said to be horizontal if it's everywhere transverse to the Seifert fibration. These foliations are very well understood, and are known to be $\mathbb{R}$-coverd. In fact, Conjecture \ref{LO} is solved affirmatively for Seifert-fibered  manifolds \cite{BGW} in the strong sense that the statement about taut foliations is also true. Now, it is also well-known that every Seifert fibered rational homology sphere can be given an orientation for which it can be represented as a negative-definite plumbing $\Gamma$. Moreover, it is also known in this case that $\mathbb{H}(\Gamma,\mathfrak{t})\simeq HF^+(\Gamma,\mathfrak{t})$ \cite{Nem2}, for all $\mathfrak{t}\in Spin^c(\Gamma)$. Since we will make heavy use of these results together, we summarise them in a single statement:

\begin{thm}\label{stipi} Suppose that $\Gamma$ is a Seifert-fibered rational homology sphere. Then, $\Gamma$ is not a lattice cohomology $L$-space $\Leftrightarrow$ $\Gamma$ is not a Heegaard Floer homology $L$-space $\Leftrightarrow$ $\pi_1(\Gamma)$ is left-orderable $\Leftrightarrow$ $\Gamma$ admits a horizontal foliation.
\end{thm}

\noindent The concept of horizontal foliation can be extended to graph manifolds in a straightforward manner, as described in \cite{BB}:

\begin{defi} Let $M$ be a graph manifold (possibly with boundary), and let $\mathcal{F}$ be a codimension-one foliation. Let $J$ denote the set of JSJ tori, and let $M$ have as boundary a collection of tori $\partial M=(T_1,...,T_k)$. Suppose that $\mathcal{F}$ is transverse to the Seifert fibration of each JSJ component of $M$, and that $\mathcal{F}\cap \big ( J\cup \partial M\big )$ is a foliation by simple closed curves, in particular with slopes $(\alpha_1,...,\alpha_k)\subset \partial M$. Then, we say that $\mathcal{F}$ is a horizontal foliation detected by $(\alpha_1,...,\alpha_k)$. 
\end{defi}

\noindent Horizontal foliations are obviously taut. Notice that our definition is slightly more restrictive than the one given in \cite{BB} (there, a horizontal foliation is not required \textit{a priori} to be detected by simple closed curves on the boundary). The nice feature of horizontal foliations on graph manifolds is that we can glue them together piece by piece, thus they are very suited for inductive arguments (provided we know how to perform the gluing). We also remark that it is shown in \cite{Brit} that horizontal foliations are $\mathbb{R}$-covered, and with nontrivial action on $Homeo^+(\mathbb{R})$. Moreover, it is shown in Lemma 5.5 of \cite{BRW} that horizontal foliations on Seifert manifolds with orientable surface underlying the base orbifold are co-orientable. Since we will be dealing exclusively with such JSJ pieces (we use bundles over the sphere as the building blocks for $\Gamma$), it follows easily that our horizontal foliations will be co-orientable. Thus, if we establish the existence of a horizontal foliation on $\Gamma$, we get at once left-orderability and exclusion from the Heegaard Floer $L$-space family.

\section{Proofs}

\subsection{Theorems A and B}

We start by making some elementary but useful remarks on the diagonalisation of $\Gamma$. Notice that we work with matrices that have a very special form: if two vertices $v$, $w$, are connected by an edge, then, up to permutation of rows and columns, near the $v$, $w$ entries the matrix $\Gamma$ looks like

$$ \Gamma=\left( \begin{array}{cccc}
* & *&*& * \\
* &-e_v & 1&* \\
* & 1 &-e_w & *\\
*&*&*&* \\
 \end{array} \right) $$

\noindent In particular, if $w, v_1,...,v_n$ is a bamboo sequence starting at a node $w$ terminating at a leaf $v_n$, with respective weights $-f, -e_1,...,-e_n$, then we may write $\Gamma$ as

$$\Gamma= \left( \begin{array}{cccccc}
\Gamma'& *&0&&...&0\\
*&  -f& 1 & 0 &... &0 \\
 0&1 &-e_1 & 1 &... &0 \\
  &0& 1 & ... & 1&... \\
 ... &... &  & 1 &-e_{n-1} &1 \\
 0&0 & 0 & ... &1 & -e_n \end{array} \right)
$$

\noindent where $\Gamma'$ is some submatrix of $\Gamma$. It is now clear that we may apply elementary row operations to partially diagonalise $\Gamma$ into the form

$$\Gamma^*= \left( \begin{array}{cccccc}
\Gamma'& *&0&&...&0\\
*&  f-\frac{1}{ [-e_1,...,-e_n]}& 0 &   &... &0 \\
 0&0 & [-e_1,...,-e_n] & 0 &... &0 \\
  & & 0 & ... & 0&... \\
 ... &... &  & 0 & [-e_{n-1},-e_n] &0 \\
 0&0 & 0 & ... &0 & -e_n \end{array} \right)
$$

\noindent where $ [-e_j,...,-e_n]$ denotes the Hirzebruch-Jung continued fraction expansion

$$ [-e_j,...,-e_n]=-e_j-\frac{1}{-e_{j+1}-\frac{1}{...-\frac{1}{-e_n}}}$$

\noindent Using this technique, we can construct an algorithm for diagonalising $\Gamma$ around a choice of preferred vertex: 

\begin{enumerate}

\item Choose a vertex $v$ in $\Gamma$, and present $\Gamma$ as an ordered tree, rooted at $v$.   
\item Starting at each leaf, apply row operations as described above to partially diagonalise $\Gamma$ into the nodes nearest to the leaves.
\item Discard the diagonalised part of this matrix, and treat the remaining matrix as the intersection matrix of some tree $\widehat{\Gamma}$ (with rational coefficients at the leaves).
\item Repeat steps 2) and 3) until $v$ is reached. Put together all the discarded bits. The result will be a diagonal matrix $\widehat{\Gamma}^{v}$.
\end{enumerate}

\noindent From elementary linear algebra, we know that negative-definiteness is preserved at each step of the algorithm. Moreover, we remark that the result of the procedure depends heavily on the choice of root $v$.

\begin{defi} Let $\Delta_{v}$ be the diagonal entry corresponding to $v$ in the matrix $\widehat{\Gamma}^{v}$. Then, the quantity $1/\Delta_v$ is said to be the \textbf{de-rationaliser} of $v$.
\end{defi}

\begin{example} Suppose we are given the following plumbing tree, and we wish to diagonalise its intersection matrix keeping the cirlced vertex as a preferred one:

\begin{center}
\includegraphics[scale=0.3]{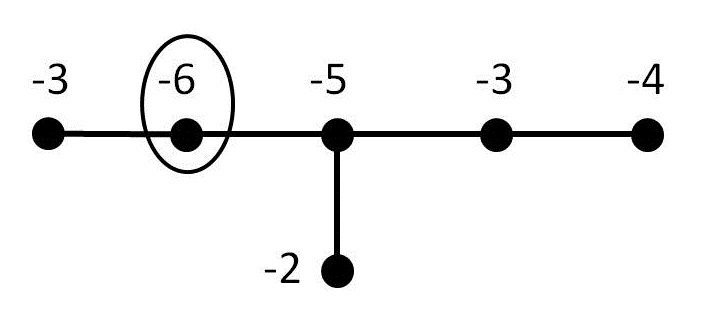}
\end{center}

The intersection matrix may be written as

$$ \Gamma=\left( \begin{array}{cccccc}
-3 & 1 & 0 &0  & 0 &0 \\
 1& -6& 0 & 1 &0 &0 \\
0 & 0 &-2 & 1 &0 &0 \\
  0&1& 1 & -5 & 1& 0 \\
0 & 0 & 0 & 1 &-3 &1 \\
0& 0 & 0 & 0 &1 & -4 \end{array} \right)
$$

Then, applying row operations, we may diagonalise $\Gamma$ in the following steps:

$$ \Gamma\longrightarrow \left( \begin{array}{cccccc}
-3 & 1 & 0 &0  & 0 &0 \\
 1& -6& 0 & 1 &0 &0 \\
0 & 0 &-2 & 0 &0 &0 \\
  0&1& 0 & -\frac{91}{22} & 0& 0 \\
0 & 0 & 0 & 0 &-\frac{11}{4} &0 \\
0& 0 & 0 & 0 &0 & -4 \end{array} \right)  \longrightarrow  \left( \begin{array}{cccccc}
-3 & 0 & 0 &0  & 0 &0 \\
 0& -\frac{1481}{273}& 0 & 0 &0 &0 \\
0 & 0 &-2 & 0 &0 &0 \\
  0&0& 0 & -\frac{91}{22} & 0& 0 \\
0 & 0 & 0 & 0 &-\frac{11}{4} &0 \\
0& 0 & 0 & 0 &0 & -4 \end{array} \right)$$

The de-rationaliser of our preferred vertex is then -273/1481.

\end{example}

\vskip 0.1in We now describe the topological significance of the de-rationaliser. Suppose that we split a plumbing tree $\Gamma$ along an edge $v-w$. We obtain two trees, each with an obvious marked vertex, and we denote these by $\Gamma^v$ and $\Gamma^w$. Topologically, this splitting induces two manifolds with torus boundary, which we also denote by $\Gamma^v$ and $\Gamma^w$. Let us consider  $\Gamma^v$, for argument's sake. Being a manifold with torus boundary, it makes sense to talk about Dehn filling on $\Gamma^v$. It is then an elementary fact that, in terms of the basis $\mu_v, \lambda_v$ described in section 2, $r\in \mathbb{Q}$ Dehn filling on $\Gamma^v$ is represented by the plumbing graph $\Gamma^{v}(r)$, which is $\Gamma^v$ augmented at $v$ by a linear chain with weights $-a_1,..., -a_k$, where $[-a_1,...,-a_k]$ is the Hirzebruch-Jung continued fraction expansion of $r$. We now prove the following lemma, which explains the name de-rationaliser:

\begin{lemma} Let $\Gamma^v$ be as above, and let $1/\Delta_v$ be its de-rationaliser. Then, $\Gamma^{v}(1/\Delta_v)$ has $b_1>0$.
\end{lemma}

\begin{proof}\label{zerosurgery} Choose $v$ as a root for $\Gamma^{v}(1/\Delta_v)$, and let's compute the entry at $v$ after the diagonalisation process. We have a contribution of $\Delta_v$ from the subgraph $\Gamma^v$, and a contribution of $-\Delta_v$ from the chain, whence this entry is zero. That is, we obtain a diagonal matrix with a zero on the diagonal. Since row operations of the type we are performing don't change the determinant of a matrix, we conclude that $det\; \Big (\Gamma^{v}(1/\Delta_v)\Big )=0$. It is a standard fact that this condition implies $b_1\Big (\Gamma^{v}(1/\Delta_v)\Big )>0$, so we're finished. 
\end{proof}

The following result will be crucial:

\begin{lemma} Let $\Gamma^v$ be a negative-definite matrix with a marked vertex, $v$. Then, the manifold $\Gamma^v(1/\Delta_v)$ admits a horizontal foliation. Moreover, the manifold with boundary $\Gamma^v$ admits a horizontal foliation detected by $1/\Delta_v$ (in terms of the canonical basis $\mu_v$, $\lambda_v$).
\end{lemma}

\begin{proof} We prove the claim by induction on the number of nodes. The one-node case is essentially an old result of Eisenbud-Hirsch-Neumann \cite{Eisen}, stating that Seifert-fibered manifolds over the sphere with $b_1>0$ admit a horizontal foliation. That $\Gamma^v(1/\Delta_v)$ is Seifert-fibered is clear from the fact that $1/\Delta_v\neq 0$, which in turn follows from negative-definiteness. Moreover, we can think of the core of the surgery solid torus as representing a fiber from this Seifert fibration. Removing a regular neighborhood of this fiber, and observing that the meridian of the surgery solid torus was represented by the slope  $1/\Delta_v$ (in terms of the canonical basis $\mu_v$, $\lambda_v$), we deduce that $\Gamma^v$ has a horizontal foliation detected by the slope $1/\Delta_v$ .
\par For the inductive step, let $\Gamma^v$ have $N>1$ nodes. Choose a terminal node $w\neq v$ in $\Gamma^v(1/\Delta_v)$,  (\textit{i.e.} a node such that deleting all its neighbours results in a disconnected graph with one component having exactly $N-1$ nodes), and delete its neighbour leading to $v$, say, $z$. We obtain a one-node graph $\Gamma^w$ and an $N-1$-node graph $\big ( \Gamma^v(1/\Delta_v)\big )^z$. Now, it is clear that $\Gamma^w(1/\Delta_w)$ admits a horizontal foliation, and that $\Gamma^w$ admits a horizontal foliation detected by $1/\Delta_w$. Now, look at $\big ( \Gamma^v(1/\Delta_v)\big )^z$. It follows from the diagonalisation algorithm described earlier in the section that $\big ( \Gamma^v(1/\Delta_v)\big )^z(\Delta_w)$ has $b_1>0$. By inductive hypothesis, this manifold has a horizontal foliation. Moreover, $\big ( \Gamma^v(1/\Delta_v)\big )^z$ has a horizontal foliation detected by the slope $\Delta_w$ (in terms of the basis $\mu_z, \lambda_z$ for its boundary component). Since the gluing matrix is $ \phi=\left( \begin{array}{cc}
0 & 1 \\
1 & 0 \\
 \end{array} \right) $, we can glue together the foliations to obtain our desired horizontal foliation. That $\Gamma^v$ admits a horizontal foliation detected by $1/\Delta_v$ is straightforward.
\end{proof}

\begin{remark} Of course, the proof of Lemma \ref{zerosurgery} generalises to plumbing matrices which are merely nonsingular. It seems plausible that one should be able to produce taut foliations detected by boundary slopes in a more general situation (at least for graph manifolds only with JSJ pieces which fiber over a surface of genus zero) using the methods developed by Gabai in the proof of the Property R conjecture \cite{Gabai}, but we decided to keep things simple and content ourselves with what is necessary.
\end{remark}

We can now establish Theorem A:

\begin{proof}[Proof of Theorem A] We proceed by induction on the number of nodes, $N$. Recall that the hypothesis is a negative-definite graph with a very bad vertex. If $N=1$, then $\Gamma$ is a Seifert-fibered space which is not a lattice cohomology $L$-space. Theorem \ref{stipi} implies the result. For the inductive step, suppose that $\Gamma$ has $N>1$ nodes. We wish to show that $\Gamma$ has a horizontal foliation. Let $v$ be a very bad node. Choose a path to an arbitrary neighbouring node, and pick the first vertex $w$ after $v$. Split $\Gamma$ at $v-w$. Consider $\Gamma^w$, and do $1/\Delta_w$-surgery. By Lemma \ref{zerosurgery}, $\Gamma^{w}$ has a horizontal foliation, detected by $1/\Delta_w$. Once again, since $ \phi=\left( \begin{array}{cc}
0 & 1 \\
1 & 0 \\
 \end{array} \right) $, the induced surgery coefficient on $\Gamma^v$ will be $\Delta_w$. We must now check that $\Gamma^{v}(\Delta_w)$ is negative-definite with a very bad vertex. Since $\Delta_w<0$, we may take its continued fraction expansion to consist exclusively of negative numbers, so that to certify that $\Gamma^{v}(\Delta_w)$ is negative-definite, it suffices to show that the following matrix is negative-definite:

$$ \left( \begin{array}{cccc}
 & & & 0 \\
 &\Gamma^v  & &... \\
 &  & & 1\\
0&...&1 &\Delta_w \\
 \end{array} \right) $$

\noindent partially diagonalising $\Gamma$ around $v$, and attacking the $\Gamma^w$ subgraph first, one obtains exactly the form written above, so that $\Gamma^{v}(\Delta_w)$ is indeed negative-definite. The only thing to worry about now is if the continued fraction expansion of $\Delta_w$ starts with $-1$, for then we must put the graph into minimal form by blowing down. Doing so will always increase the weight of $v$ by one, so that it definitely remains very bad even in this situation.  We can now say that, by inductive hypothesis, $\Gamma^{v}(\Delta_w)$ has a horizontal foliation. Moreover, it is easy to see that $\Gamma^{v}$ has a horizontal foliation detected by $\Delta_w$. Hence, we may glue together the foliations on $\Gamma^v$ and $\Gamma^w$, and the result is established. 
\end{proof}

We now demonstrate Theorem B:

\begin{proof}[Proof of Theorem B] Once again, we proceed by induction on the number of nodes. The base case is an application of Proposition \ref{lattice}, together with Theorem \ref{stipi}. For the inductive step, we suppose once again that $\Gamma$ has $K>1$ nodes, and a proper $E_8$ subgraph. The only thing that has to be checked is whether the surgery step after we cut still gives us a proper $E_8$ subgraph. Recall the labelled $E_8$ graph from Figure 1. There are 4 cases to consider, depending on where we find a piece of $\Gamma$ abutting from the $E_8$ subgraph.

\vskip 0.1in \noindent \textit{Case 1 abutting from $X$:} This case is the easiest. If there is abutting from $X$, then $\Gamma$ has a very bad vertex, and the conclusion of the theorem follows at once from Theorem A.

\vskip 0.1in
\noindent \textit{Case 2 abutting from the chain $x_1-...-x_4$:} Our strategy will be to cut at the edge $X-x_1$. In doing so, we are stealing from the $E_8$ subgraph, and so we must ensure that the continued fraction expansion of the surgery coefficient for $\Gamma^X$, which is $\Delta_{x_1}$, starts with four $-2$ and does not strictly end there, in order to obtain a proper $E_8$ subgraph for the inductive step. To verify this, we estimate the de-rationaliser of $\Gamma^{x_1}$ at $x_1$. First of all, notice that $\Gamma$ contains the following sub-configuration, and this must be negative definite.

\begin{center}
\includegraphics[scale=0.30]{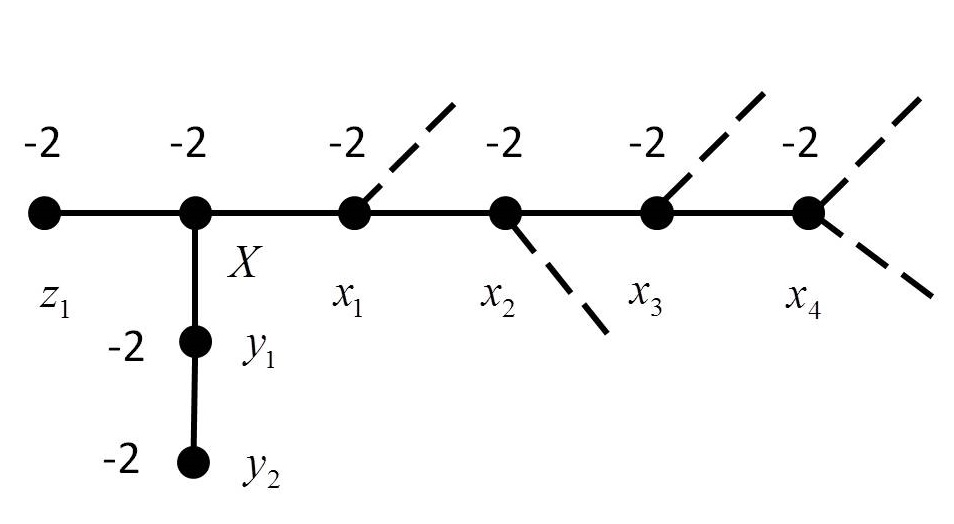}
\end{center}

\noindent Partially diagonalising the graph above into $x_1$, attacking the side which contains $X$, we calculate that $-\frac{5}{6}$ surgery on $\Gamma^{x_1}$ is still negative-definite. This places a lower bound of $-\frac{5}{6}$ for the de-rationaliser of $\Gamma^{x_1}$. To get an upper bound, we diagonalise the subgraph $\Gamma^{x_1}$ into $x_1$. Since the contribution coming into $x_4$ is non-negative, we get that $x_3$ receives a contribution $r_{x_4}\geq \frac{1}{2}$ from $x_4$. The same reasoning implies that $x_2$ now receives a contribution $r_{x_3}\geq \frac{1}{2-\frac{1}{2}}= \frac{2}{3}$ from $x_3$. Finally, $x_1$ receives a total contribution from its neighbours of $r> \frac{3}{4}$, where the inequality is strict since we are assuming that there are some vertices abutting from the chain $x_1-...-x_4$. Since we want to obtain a $0$ at $x_1$ when we add the contribution coming from the de-rationaliser, we immediately obtain the upper bound $1/\Delta_{x_1}<\frac{1}{-2+\frac{3}{4}}=-\frac{4}{5}$. Thus, $-\frac{5}{4}<\Delta_{x_1}<-\frac{6}{5}$. It is straightforward to verify, given the bounds above, that the Hirzebruch-Jung continued fraction expansion of $\Delta_{x_1}$ starts with four $-2$ and does not strictly end there, so that $\Gamma^X(\Delta_{x_1})$ contains a proper $E_8$ subgraph, and the inductive step of the argument can go through.

\vskip 0.1in
\noindent \textit{Case 3 abutting from the chain $y_1-y_2$:} We cut at $y_1-X$. This time, we must show that the continued fraction expansion of $\Delta_{y_1}$ starts with two $-2$ and does not strictly end there, to obtain a proper $E_8$ subgraph. Here, the lower bound for the de-rationaliser is computed to be $-\frac{7}{10}$. The upper bound is found to be $-\frac{2}{3}$, whence $-\frac{3}{2}<\Delta_{y_1}<-\frac{10}{7}$, which is good enough so that we may proceed.

\vskip 0.1in
\noindent \textit{Case 4 abutting from $z_1$:} We cut at $z_1-X$. Now, we only need the continued fraction to start with $-2$ and not strictly end there. Using similar methods as above, we obtain the estimates $-2<\Delta_{z_1}<-\frac{15}{8}$, so that the inductive step can go through in this case.

\vskip 0.1in 
\par We've just verified that the inductive argument can proceed in all situations. Now, the horizontal foliation argument in the proof of Theorem A works verbatim for Theorem B (with the property of having a very bad vertex replaced with the property of having a proper $E_8$ subgraph), so we are finished.
\end{proof}

\begin{remark} In \cite{BB}, the authors establish Conjecture \ref{LO} for graph-manifold integer homology spheress by successfully constructing a horizontal foliation on any graph-manifold integral homology sphere not homeomorphic to $S^3$ or the Poincar\'e homology sphere $\Sigma(2,3,5)$ (such a foliation cannot exist in these two cases). Unfortunately, it is not hard to construct graph-manifold rational homology spheres without horizontal foliations (in the stronger sense of our definition), so that the strategy of proof of theorems A and B cannot work in general. Consider the negative-definite plumbing $\Gamma$ below.

\begin{center}
\includegraphics[scale=0.25]{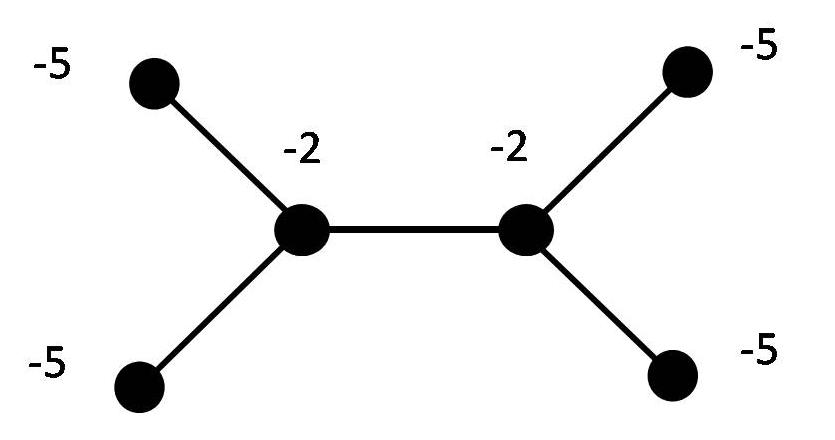}
\end{center}

\noindent This has a pair of bad vertex neighbours, so a quick check with Laufer's algorithm shows that $\Gamma$ is not a lattice cohomology $L$-space. Moreover, it is even known in this case that $\Gamma$ is not a Heegaard Floer $L$-space \cite{Nem3} (we are thankful to Andr\'as N\'emethi for pointing this out). We claim that $\Gamma$ has no horizontal foliations. To see this, we split $\Gamma$ along the central edge to obtain two identical graphs $\Gamma^v$, and we will try to build the foliation by gluing it along the Seifert-fibered pieces. Notice, by blowing down, that $\Gamma^v(-1)$ is negative-definite, in fact, homeomorphic to a lens space. Hence, any surgery $\Gamma^v(r)$ with $r\leq -1$ gives a manifold which is not a lattice cohomology $L$-space, hence without horizontal foliations, by Theorem \ref{stipi}. If we do $r\in \mathbb{Q}^+$ surgery, we need to massage $\Gamma^v(r)$ into negative-definite form. Using the plumbing calculus, it is easy to show that this involves decreasing the weight of the central node by $\lceil 1/r \rceil$ units, and extending by the continued fraction expansion of $-(1/r-\lceil 1/r \rceil)^{-1}$. Since $1/r-\lceil 1/r \rceil<1$, the continued fraction expansion of $-(1/r-\lceil 1/r \rceil)^{-1}$ does not start with $-1$. We see that $\Gamma^v(r)$ can be represented by a minimal negative-definite graph with no bad vertices, so that we cannot put a horizontal foliation on it, by Theorem \ref{stipi}. In addition, $0$ surgery gives a connected sum of lens spaces (this can be seen by plumbing calculus, or just by standard Seifert manifold theory), and this is also no good. We conclude that $\Gamma^v$ cannot have a horizontal foliation detected by any slope in $\mathbb{Q}-(-1,0)$. Finally, since the gluing matrix $\phi$ takes $r$ to $1/r$, there is no need to study the case $r\in (-1,0)$. Thus, it is impossible for $\Gamma$ to admit a horizontal foliation.
\end{remark}

\subsection{Theorem C}

We now turn to the proof of Theorem C. In order to do so, we first describe a presentation for the fundamental group of a plumbed $3$-manifold, due to Mumford \cite{Mumford}:

\begin{thm} Let $\Gamma$ be a be a plumbing along a graph, with vertices $\{v_i\}_{i=1}^N$ and suppose we take only circle bundles over a sphere. Denote the set of edges of the graph by $E(\Gamma)$. Given a vertex $v$, choose a cyclic ordering of its neighbours.  Let $N(v)$ denote the cyclically ordered set of neighbours of $v$. Then, $\pi_1(\Gamma)$ admits the following presentation:

$$ \pi_1(\Gamma)=\Big \langle v_i \;  \Big | \; v_i^{-w(v_i)}=\prod_{v_k\in N(v_i)}v_k \;,\; [v_j,v_k]=1 \;if\; v_j-v_k\in E(\Gamma)           \Big \rangle
$$
\end{thm}

\begin{remark} Even though picking different orderings of the neighbours results necessarily in presentations describing isomorphic groups, we are not allowed to change the cyclic ordering once we've picked a particular presentation.
\end{remark}

Let us recall the hypotheses of Theorem C: we have a negative-definite insulated graph $\Gamma$.  We will argue by contradiction, so from now until the end of the proof we will suppose that $\pi_1(\Gamma)$ admits some left-ordering $>$. Clearly, we may assume that $\Gamma$ has at least two vertices (otherwise, $\pi_1$ is finite and the result is trivial). Now, arbitrarily pick a vertex $Y$, and present $\Gamma$ as an ordered tree, rooted at $Y$. This naturally induces a notion of height along the tree, measured by distance to $Y$. For notational convenience, we will take $Y$ to be at the ``highest point''. For a given vertex $X$, we say its descendents are its neighbours at a lower height, and its parent is the (unique) neighbour at higher height.

\vskip 0.1in \noindent We need to prove two preliminary lemmata:

\begin{lemma}\label{badnode} Let $\Gamma$ be a rooted tree. Suppose that $X$ is a bad node, with descendents $x_1,..., x_n$, and parent $z$. If we have that $X\geq x_i^2>1$ $\forall i$, then $z^2>X$.
\end{lemma}

\begin{proof} We start by pointing out that it is also true that $x_i^2>x_i$, so that $X>x_i>1$. If $z\geq X>1$, the result is immediate: $z>1\Rightarrow z^2>z$, so that $z^2>X$. So we may take $X>z$. Now, write the group relation at the node $X$. Since neighbours commute, we may write this as

$$ X^{K}z^{-1}(x_1^{-1}X)...(x_n^{-1}X)=1
$$

\noindent where $K\geq 0$ (this uses the fact that the vertex is not very bad). Now, by conjugating, we may permute cyclically the order of the neighbours. Therefore, we may write

$$ X^{K}(x_3^{-1}X)...(z^{-1}X)(x_1^{-1}X)x_2^{-1}=1
$$

\noindent Now, $X\geq x_2^2$ implies that $x_2^{-1}\geq X^{-1}x_2$. We may substitute in the above (it is left as an exercise to the reader to verify that this can be done with left-orderings) to get

$$ 1\geq  X^{K}(x_3^{-1}X)...(z^{-1}X)(x_1^{-1}X)x_2X^{-1}=X^{K}(x_3^{-1}X)...(z^{-1}X)(x_1^{-1}x_2)
$$

\noindent Now, since all the factors in brackets with $X$ are $>1$, we must have the strict inequality $1>x_1^{-1}x_2\Rightarrow x_1>x_2$. Cyclically permuting more and applying the same reasoning with the appropriate relations one obtains the chain

$$z>x_1>x_2>...>x_n
$$

 \noindent Now suppose that the conclusion of the lemma is false. Then, $1\geq X^{-1}z^2\Rightarrow z^{-1}\geq X^{-1}z$. The same reasoning as above shows that $x_n>z$, which is a contradiction. Therefore, $z^2>X$.
\end{proof}

\noindent Having just done bad nodes, we now investigate how the ordering behaves as we go up the rooted tree via a good vertex:

\begin{lemma}\label{goodnode} Let $\Gamma$ be an insulated rooted tree, and let $X$ be a good vertex. Partition its descendents into bad vertices $y_1,..,y_k$ and good vertices $x_1,..,x_{l}$. Suppose that $X^2>y_i>1$, and that $X>x_i>1$. Let $z$ be the parent neighbour. Then,

\begin{enumerate}

\item $z>X$, if $z$ is a good vertex.
\item $z>X^2$, if $z$ is a bad vertex.
\end{enumerate}
\end{lemma}

\begin{proof} We only do the first item, the second one being very similar. Suppose that $z$ is a good vertex. Write the relation at $X$ as

 $$X^N\prod_{i=1}^{k}y_i^{-1}\prod_{j=1}^{l-1}x_j^{-1}z^{-1}=1 \Leftrightarrow$$

$$ X^{N-2k-l}\prod_{i=1}^{k}(y_i^{-1}X^2)\prod_{j=1}^{l-1}(x_j^{-1}X)(z^{-1}X)=1$$

\noindent It seems that we have forced a specific ordering of the vertices around $X$, but it will be seen from the argument that no loss of generality is incurred by doing this (and anyway, we are allowed to pick the ordering once). Now, $N-2k-l\geq0$, from the insulation hypothesis. This, together with the other hypothesis of the lemma, implies that 

$$ X^{N-2k-l}\prod_{i=1}^{k}(y_i^{-1}X^2)\prod_{j=1}^{l-1}(x_j^{-1}X)>1$$

\noindent Therefore, $1>z^{-1}X$, and the result follows.

\end{proof}

\begin{defi} If a vertex $X$ satisfies either Lemma \ref{badnode} or \ref{goodnode}, we say it satisfies condition $A$.
\end{defi}

\noindent With these ingredients in place, we can now prove Theorem C:

\begin{proof}[Proof of Theorem C] Let $\Gamma$ be an insulated graph, rooted at some node $Y$. Suppose that $\pi_1(\Gamma)$ is left-orderable. We will argue inductively along the height of the tree $\Gamma$. We basically want to prove that all vertices satisfy condition $A$. First, we claim the following:

\begin{claim2*} Let $y$ be a leaf of $\Gamma$. Then, the element of $\pi_1(\Gamma)$ determined by $y$ is nontrivial.
\end{claim2*}

\noindent To prove this, suppose that $y$ is trivial. Then, $y$ bounds a disk in $\Gamma$, which we can take to be embedded, by Dehn's Lemma. Notice that $y$ represents one of the exceptional fibers of one of the JSJ components of $\Gamma$. If $\Gamma$ is a Seifert-fibered space (\textit{i.e.} the graph $\Gamma$ has one node), it is well-known that elements represented by exceptional fibers are nontrivial in the fundamental group (unless, of course, $\Gamma$ is $S^3$, which is precluded by hypothesis). If $\Gamma$ has more than one JSJ component, then the disk bounded by $y$ must intersect the corresponding boundary component, by the same reason as above. This intersection is a collection of parallel closed loops. A standard innermost loop argument now implies that one of these loops must bound a disk, contradicting the incompressibility of the boundary. Therefore, $y$ must be nontrivial in $\pi_1(\Gamma)$.

\vskip 0.1in \noindent Now, it is straightforward to verify that if we reverse the inequality signs in the hypotheses of Lemmas \ref{badnode} and \ref{goodnode}, it is possible to reach analogous conclusions, but also with inequality signs reversed (by reversing inequality signs, we mean permuting right and left-hand sides, of course). If a vertex satisfies either of these ``reversed inequality'' conditions, then we say it satisfies condition $A^*$. We then make the following claim:

\begin{claim3*} Each vertex of $\Gamma$ satisfies either condition $A$ or condition $A^*$ (except for the root $Y$ and leaves, for which the conditions are vacuous). 
\end{claim3*}

\noindent We prove this by induction on the height of the rooted tree. First, it is straightforward to verify this for height equal to $1$: let $X$ be a vertex at height $1$, and observe that it has at least one leaf neighbour $y$ at height $0$. If $X$ were trivial in $\pi_1$, then the group relations would imply that $X=y^n$, for some $n\geq 2$. This would imply that $y$ is torsion. Since, for irreducible $3$-manifolds, having infinite fundamental group implies having torsion-free fundamental group, we see that $X$ cannot be trivial. It follows immediately that $X$ satisfies either condition $A$ or $A^*$. For the inductive step, suppose that all vertices at height $\leq k$ satisfy either of the conditions. Let $Z$ be a vertex at height $k+1$. We observe that all the descendents of $Z$ must satisfy the same condition. If that weren't the case, we could show that $1>Z>1$, which is a contradiction. Hence, without loss of generality, we suppose that the descendents all satisfy condition $A$. If $Z$ is bad, by insulation hypothesis, all its descendents are good. The ones that have descendents themselves, say $x_1,...,x_k$, satisfy Lemma \ref{goodnode} by inductive hypothesis, so we have indeed $Z>x_i^2$. The descendents that don't have descendents themselves, say $y_1,...,y_m$, are leaves, and so we have trivially that $Z\geq y_i^2$, bearing in mind that all weights are at most $-2$, since we take $\Gamma$ to be minimal . We conclude that $Z$ satisfies condition $A$. If $Z$ is good, with bad descendents $x_1,...,x_p$ and good descendents $y_1,...,y_q$, we clearly have that $Z^2>x_i$, since the $x_i$ satisfy Lemma \ref{badnode} by inductive hypothesis, and also that $Z>y_i$. Hence, $Z$ also satisfies condition $A$, and the claim is established.

\vskip 0.1in \noindent Notice that it follows in hindsight that all vertices of $\Gamma$ satisfy the same condition out of $A$ and $A^*$, since $\Gamma$ is connected. Without loss of generality, suppose that condition $A$ holds everywhere. This implies that the root $Y$ is $>$-maximal among the generators in the group presentation. Now, presenting $\Gamma$ as a rooted tree by choosing a different vertex for the root implies that $Y$ is not $>$-maximal, and this is a contradiction. Therefore, $\pi_1(\Gamma)$ cannot be left-orderable. That $\Gamma$ is a  Heegaard Floer $L$-space is just Corollary \ref{lati}. We are finished.

\end{proof}


\begin{thebibliography}{99}

\bibitem{BB} Boileau, Michel; Boyer, Steven, \textit{Graph manifolds $\mathbb{Z}$-homology spheres and taut foliations}, arXiv:1303.5264.

\bibitem{BGW} Boyer, Steven; Gordon, Cameron McA.; Watson, Liam;  \textit{On L-spaces and left-orderable fundamental groups}, Math. Ann. 356 (2013), no. 4, 1213–1245.

\bibitem{BRW} Boyer, Steven; Rolfsen, Dale; Wiest, Bert, \textit{Orderable 3-manifold groups},  Ann. Inst. Fourier (Grenoble) 55 (2005), no. 1, 243–288.

\bibitem{Brit} Brittenham, Mark, \textit{Tautly foliated 3-manifolds with no $\mathbb{R}$-covered foliations}, Foliations: geometry and dynamics (Warsaw 2000), 213-224, World Sci. Publ., River edge, NJ, 2002.

\bibitem{CLW} Clay, Adam; Lidman, Tye; Watson, Liam, \textit{Graph manifolds, left-orderability and amalgamation}, Alg. Geom. Topology 13 (2013), 2347–2368.

\bibitem{Eisen} Eisenbud, David; Hirsch, Ulrich; Neumann, Walter, \textit{Transverse foliations of Seifert bundles and self homeomorphism of the circle}, Comment. Math. Helvetici. 56 (1981), 638-660. 

\bibitem{Gabai} Gabai, David; \textit{Foliations and the topology of 3-manifolds}, J. Differential Geom. 26 (1987), no. 3, 461-478.

\bibitem{Laufer} Laufer, Henry B., \textit{On rational singularities}, American Journal of Mathematics, vol. 94, no.2 (1972), 597-608.

\bibitem{Li} Li, Tao; Roberts, Rachel, \textit{Taut foliations in knot complements}, arXiv:1211.3066.

\bibitem{Mumford} Mumford, David, \textit{Topology of Normal Singularities and a Criterion for Simplicity}, Publ. de l'Institut des Hautes Etudes Scientifiques (1961), pp. 5-22.

\bibitem{Nem2} Némethi, András, \textit{On the Ozsv\'ath- Szab\'o invariant of negative-definite plumbed 3-manifolds},  Geometry and Topology 9 (2005), 991-1042.

\bibitem{Nem1} Némethi, András, \textit{Lattice cohomology of normal surface singularities}, Publ. Res. Inst. Math. Sci. 44 (2008), no. 2, 507–543.

\bibitem{Nem3} Némethi, András; Rom\'an, Fernando, \textit{The lattice cohomology of $S^3_{-d}(K)$}, in \textit{Zeta functions in algebraic geometry}, Contemporary Math., 566 (2012), 261-292.

\bibitem{Ni} Ni, Yi; \textit{Knot Floer homology detects fibred knots}, Invent. Math. 170 (2007), no. 3, 577-608.

\bibitem{OS2} Ozsv\'ath, Peter; Szab\'o, Zolt\'an, \textit{Holomorphic disks and topological invariants for closed three-manifolds}, Ann. of Math. (2) 159 (2004), 1027-1158. 

\bibitem{OS1} Ozsv\'ath, Peter; Szab\'o, Zolt\'an, \textit{Holomorphic disks and genus bounds}, Geom. Topol. 8 (2004), 311-334.


\bibitem{Tran} Tran, Ahn, \textit{On left-orderable fundamental groups and Dehn-surgeries on knots}, arXiv:1301.2637.

\end{thebibliography}
\end{document}